\pdfoutput=1
\newif\ifpreprint
\preprinttrue   

\ifpreprint

	\documentclass[a4paper,12pt]{article}

\usepackage{abstract} 

\usepackage{parskip}

\usepackage{titlesec}
\def\sectionfont{\sffamily\Large\bfseries\boldmath}
\def\subsectionfont{\sffamily\large\bfseries\boldmath}
\def\paragraphfont{\sffamily\normalsize\bfseries\boldmath}
\titleformat*{\section}{\sectionfont}
\titleformat*{\subsection}{\subsectionfont}
\titleformat*{\subsubsection}{\paragraphfont}
\titleformat*{\paragraph}{\paragraphfont}
\titleformat*{\subparagraph}{\paragraphfont}
\usepackage[small,labelfont={bf,sf}]{caption}  
\usepackage{fullpage,graphicx,psfrag,url}
\usepackage{multirow}
\usepackage{enumitem}
\setlist{nolistsep}

\usepackage{booktabs}   
\usepackage{adjustbox}  

\newcounter{algorithmctr}[section]
\renewcommand{\thealgorithmctr}{\thesection.\arabic{algorithmctr}}

\usepackage{amsmath,amsthm}

\newtheoremstyle{exampstyle}
  {.5\baselineskip} 
  {\topsep} 
  {} 
  {} 
  {\bfseries} 
  {.} 
  {.5em} 
  {} 
\theoremstyle{exampstyle}

\newtheorem{lemma}{Lemma}[section]
\newtheorem{proposition}{Proposition}
\newtheorem{remark}{Remark}

\newtheorem{assumption}{Assumption}

\usepackage[english]{babel}
\usepackage{algorithm}
\usepackage[noend]{algpseudocode}
\usepackage{amssymb,mathtools,amsfonts}
\algrenewcommand\algorithmicindent{1em}%

\usepackage{hyperref}
\hypersetup{%
    colorlinks=true, 
    linktoc=all,     
    linkcolor=black,  
    citecolor= black,
    urlcolor=blue
}
\usepackage[xindy, acronyms]{glossaries}   

\newcommand{\eg}{\emph{e.g.}\ }
\newcommand{\ie}{\emph{i.e.}\ }
\newcommand{\iid}{i.i.d.\ }

\newcommand{\mcf}{\mathcal}

\renewcommand{\Re}{{\mbox{\bf R}}}
\newcommand{\ReExt}{\tilde{\mbox{\bf R}}}
\newcommand{\Nat}{\bf N}

\newcommand{\eqdef}{\coloneqq}
\newcommand{\eps}{\varepsilon}

\newcommand{\Alg}{{Alg.}}
\newcommand{\App}{{Appendix}}

\newcommand{\Cor}{{Cor.}}

\newcommand{\Fig}{{Figure}}

\newcommand{\Prop}{{Prop.}}
\newcommand{\Thm}{{Thm.}}

\DeclareMathOperator*{\argmin}{\operatorname{argmin}}

\DeclareMathOperator{\prox}{prox}

\newcommand{\project}[1]{{\rm\Pi}_{#1}}
\newcommand{\dist}[1]{{\rm dist}_{#1}}

\newcommand{\normalCone}[1]{N_{#1}}
\newcommand{\polar}[1]{{#1}^\circ}

\newcommand{\abs}[1]{|#1|}
\newcommand{\norm}[1]{\lVert#1\rVert}
\newcommand{\normbig}[1]{\big\lVert#1\big\rVert}
\newcommand{\half}{\tfrac{1}{2}}
\newcommand{\innerprod}[2]{\left\langle{#1},{#2}\right\rangle}
\newcommand{\seq}[1]{\lbrace {#1} \rbrace_{k\in\Nat}}

\usepackage[binary-units=true]{siunitx}
\usepackage{tikz} 
\usepackage{pgfplots}
\pgfplotsset{compat=newest} 
\usepgfplotslibrary{units} 
\pgfplotsset{%
    /pgfplots/ybar legend/.style={
    /pgfplots/legend image code/.code={%
       \draw[##1,/tikz/.cd,
       yshift=-0.3em
       ]
        (0cm,0cm) rectangle (5pt,0.8em);},
   }
}
\pgfplotsset{select coords between index/.style 2 args={
    x filter/.code={
        \ifnum\coordindex<#1\fi
        \ifnum\coordindex>#2\fi
    }
}}

\newacronym{ADMM}{ADMM}{alternating direction method of multipliers}
\newacronym{BPD}{BPD}{basis pursuit denoising}
\newacronym{DPDA}{DPDA}{distributed primal-dual algorithm}
\newacronym{QP}{QP}{quadratic program}
\newacronym{SOCP}{SOCP}{second-order cone program}

\newcommand{\OptProblem}[5][]{
\begin{aligned}\label{#1}
  \begin{array}{ll}
    \underset{#3}{\rm{#2}}&\hspace{-0ex}#4\vspace{.5ex}\\
    \rm{subject~to}&#5
  \end{array}
\end{aligned}
}

\newcommand{\MinProblem}[4][]
{\OptProblem[#1]{minimize}{#2}{#3}{#4}}

\usepackage{xcolor}

\else

	\documentclass[letterpaper,10pt,conference]{ieeeconf}	

	\IEEEoverridecommandlockouts							
	\useRomanappendicesfalse

	\input defs_ieee.tex
\fi

\title{%
	\ifpreprint
		\bfseries\sffamily
	\else
		\bfseries\LARGE
	\fi
	Decentralized Resource Allocation via Dual Consensus ADMM
}

\author{%
Goran Banjac,
Felix Rey,
Paul Goulart,
and John Lygeros%
\ifpreprint
\else
	\thanks{%
		G.\ Banjac, F.\ Rey, and J.\ Lygeros are with the Automatic Control Laboratory, ETH Zurich, Physikstrasse 3, 8092 Zurich, Switzerland.
		{\tt\footnotesize \{gbanjac, rey, lygeros\}@control.ee.ethz.ch}}%
	\thanks{%
		P.\ Goulart is with the Department of Engineering Science, University of Oxford, Oxford OX1~3PJ, UK.
		{\tt\footnotesize paul.goulart@eng.ox.ac.uk}}%
\fi
}

\begin{document}
\maketitle

\begin{abstract}
	We consider a resource allocation problem over an undirected network of agents, where edges of the network define communication links.
	The goal is to minimize the sum of agent-specific convex objective functions, while the agents' decisions are coupled via a convex conic constraint.
	We derive two methods by applying the \gls{ADMM} for decentralized consensus optimization to the dual of our resource allocation problem.
	Both methods are fully parallelizable and decentralized in the sense that each agent exchanges information only with its neighbors in the network and requires only its own data for updating its decision.
	We prove convergence of the proposed methods and demonstrate their effectiveness with a numerical example.
\end{abstract}

\glsresetall

\section{Introduction}

Solving optimization problems in a distributed fashion has attracted increased attention in many research areas.
This is mainly motivated by the rapid growth in size and complexity of modern datasets, which makes them hard (or even impossible) to process on a single computational unit \cite{Boyd:2011}.
On the other hand, optimization problems arising in multi-agent systems usually have a separable structure making distributed optimization methods a natural choice for solving them \cite{Banjac:2018}.
Even if such problems were solvable in a centralized fashion, the agents would need to share their local data and objective functions with the central coordinator, which would then raise information privacy issues \cite{Deori:2016}.

Distributed optimization methods are based on an iterative procedure in which the agents perform local computations and share information with other agents through a communication protocol which is often defined on a connected graph (network) \cite{Xiao:2006}.
While in some methods the agents require global information about the graph, such as the overall number of nodes or the graph Laplacian \cite{Xiao:2004}, we will focus on those in which the agents do not require a central coordinator or any global information about the graph.

\subsection*{Problem description}

Let $\mcf{G}=(\mcf{N},\mcf{E})$ denote a graph of $N\in\Nat$ agents, where $\mcf{N}\eqdef \{ 1,\ldots,N \}$ is the set of nodes, and $\mcf{E}\subseteq\mcf{N}\times\mcf{N}$ is the set of edges.
Suppose that node $i\in\mcf{N}$ can send information to node $j\in\mcf{N}$ only if $(i,j)\in\mcf{E}$.

Consider the following resource allocation problem:
\begin{equation}
	\MinProblem[{eqn:main}]{x}{\displaystyle \sum_{i\in\mcf{N}} f_i(x_i)}{\displaystyle \sum_{i\in\mcf{N}} (A_i x_i - b_i) \in \mcf{K},}
	\tag{$\mcf{P}$}
\end{equation}
where $x_i\in\Re^{n_i}$, $x\in\Re^n$ is obtained by vertically concatenating vectors $x_i$ for all $i\in\mcf{N}$, and $n=\sum_{i\in\mcf{N}} n_i$.
Problems of this form arise in numerous research areas including network flow control \cite{Bertsekas:1998}, communication networks \cite{Shen:2012}, signal processing \cite{Chen:1998}, and economics \cite{Heal:1969}.

We are interested in solving \ref{eqn:main} in a parallel and decentralized fashion so that only neighbor-to-neighbor communications are allowed.
Each node $i\in\mcf{N}$ has access only to its local objective function $f_i:\Re^{n_i}\mapsto\ReExt$, as well as $A_i\in\Re^{m\times n_i}$, $b_i\in\Re^m$, and $\mcf{K}\subseteq\Re^m$.
We make the following assumptions throughout the paper:
\begin{assumption}\label{ass:problem}~
	\begin{enumerate}[label=(\roman*)]
		\item $f_i$ is convex, closed, and proper for all $i\in\mcf{N}$.
		\item $\mcf{K}$ is a nonempty, closed, and convex cone.
		\item A primal-dual solution exists and the duality gap is zero.
		\item $\mcf{G}$ is a connected undirected graph.
	\end{enumerate}
\end{assumption}

We make no additional assumptions on the problem such as differentiability of the objective functions, or full rank of the constraint matrices.
Note that we allow each agent to have individual convex constraints of the type $x_i \in \mcf{X}_i$, where $\mcf{X}_i \subseteq \Re^{n_i}$ is a nonempty, closed, and convex set, which can be incorporated in the objective by adding the indicator function~$\mcf{I}_{\mcf{X}_i}$ to $f_i$.
Also, observe that \ref{eqn:main} allows for multiple constraints with possibly different cones, which can be cast as a single constraint using the Cartesian product of the cones.
Since the graph is undirected, $(i,j)\in\mcf{E}$ implies $(j,i)\in\mcf{E}$.

\subsection*{Related work}

The \gls{ADMM} was shown to be very effective for solving large-scale optimization problems in a distributed fashion \cite{Boyd:2011}, and many variations of the algorithm have been proposed \cite{Bertsekas:1997,Wei:2012,Wei:2013,Chang:2016}.
The authors in \cite{Deng:2017} use a Jacobi-like \gls{ADMM} for solving a variant of \ref{eqn:main} in which the computations are decomposed into $N$ smaller subproblems.
The algorithm is centralized because each node in the graph shares its decision vector with a central coordinator which then broadcasts updated information back to the nodes.
However, the existence of such a central coordinator may be undesirable in some applications.

The authors in \cite{Chang:2015} use the dual consensus \gls{ADMM} for solving a subclass of \ref{eqn:main} in which $\mcf{K}=\{0\}$.
The algorithm is fully decentralized and each node updates its decision vector based only on its own data and neighbor communications, but can handle only coupling constraints described by linear equalities, which limits applicability of the method.
The authors in \cite{Aybat:2016,Aybat:2016:arxiv} propose the \gls{DPDA}, which is based on an algorithm studied in \cite{Chambolle:2016}.
The algorithm consists of simple iterations and converges under certain choices of algorithm parameters, which can be computed based on local information from each agent.

In this paper we propose two methods based on \gls{ADMM} which can be seen as extensions of \cite[\Alg~3]{Chang:2015} for solving \ref{eqn:main} with $\mcf{K}$ being a general nonempty, closed, and convex cone.
We prove convergence of the proposed methods and demonstrate via a numerical example that both methods outperform \gls{DPDA} in terms of the iteration complexity.

\subsection*{Notation}

Let $\Nat$ denote the set of natural numbers, $\Re$ the set of real numbers,
$\ReExt\eqdef\Re\cup\{+\infty\}$ the extended real line, and $\Re^n$ the $n$-dimensional real space equipped with an inner product $\innerprod{\cdot}{\cdot}$ and induced norm $\norm{\cdot}$.
We denote by $\Re^{m\times n}$ the set of real $m$-by-$n$ matrices.
The adjoint to a linear operator $A:\Re^n\mapsto\Re^m$ is defined as the unique operator $A^*:\Re^m\mapsto\Re^n$ that satisfies $\innerprod{Ax}{y}=\innerprod{x}{A^*y}$.
We denote by $(x_i)_{i\in\mcf{N}}$ the vector obtained by vertical concatenation of vectors $x_i$, and by $[A_i]_{i\in\mcf{N}}$ the matrix obtained by horizontal concatenation of matrices $A_i$ for all $i\in\mcf{N}$.

The \emph{conjugate} of a convex, closed, and proper function $f: \Re^n\mapsto\ReExt$ is given by $f^*(y) \eqdef \sup_x \left\lbrace \innerprod{y}{x} - f(x) \right\rbrace$,
\ifpreprint
	the \emph{subdifferential} of $f$ by $\partial f(x) \eqdef \lbrace u\in\Re^n \mid (\forall y\in\Re^n) \: \innerprod{y-x}{u} + f(x) \le f(y) \rbrace$,
\fi
and the \emph{proximal operator} of $f$ by $\prox_f^\rho(x) \eqdef \argmin_y\lbrace f(y) + \tfrac{\rho}{2}\norm{y-x}^2 \rbrace$ where $\rho>0$ is a parameter.

For a nonempty, closed, and convex set~$\mcf{C}\subseteq\Re^n$ we denote its \emph{indicator function} by $\mcf{I}_\mcf{C}$ (which takes value $0$ if its argument $x\in\Re^n$ belongs to $\mcf{C}$ and $+\infty$ otherwise), the \emph{distance} of~$x\in\Re^n$ to~$\mcf{C}$ by $\dist{\mcf{C}}(x) \eqdef \min_{y\in\mcf{C}} \norm{x-y}$,
\ifpreprint
	the \emph{projection} of~$x\in\Re^n$ onto~$\mcf{C}$ by $\project{\mcf{C}}(x) \eqdef \argmin_{y\in\mcf{C}} \norm{x-y}$, and the \emph{normal cone} of~$\mcf{C}$ at $x \in\mcf{C}$ by $\normalCone{\mcf{C}}(x) \eqdef \lbrace u\in\Re^n \mid \sup_{u\in\mcf{C}}\innerprod{u}{y-x}\le 0 \rbrace$.
	Note that~$\project{\mcf{C}}$ and $\normalCone{\mcf{C}}$ are the proximal operator and the subdifferential of~$\mcf{I}_\mcf{C}$, respectively.
\fi
For a convex cone~$\mcf{K}\subseteq\Re^n$, we denote its \emph{polar cone} by $\polar{\mcf{K}} \eqdef \lbrace y\in\Re^n \mid \sup_{x\in\mcf{K}}\innerprod{x}{y}\le 0 \rbrace$.

For a graph $\mcf{G}=(\mcf{N},\mcf{E})$, let $\mcf{N}_i \eqdef \{ j \in \mcf{N} \mid (i,j) \in \mcf{E} \}$ denote the set of neighboring nodes of node $i\in\mcf{N}$, and $d_i\eqdef\abs{\mcf{N}_i}$ its degree.

\section{Dual consensus ADMM}

\gls{ADMM} is an operator splitting method that can be used to solve structured optimization problems \cite{Boyd:2011}.
Due to its relatively low per-iteration computational cost and ability to decompose an optimization problem into a sequence of smaller problems, the method is suitable for distributed and large-scale optimization \cite{Boyd:2011,Iutzeler:2016}.

The authors in \cite{Mateos:2010} propose two variants of \gls{ADMM} that can be used to solve the following consensus optimization problem over a connected undirected graph:
\begin{equation}\label{eqn:primal}
	\textrm{minimize} \quad \sum_{i\in\mcf{N}} \psi_i(y),
\end{equation}
where $\psi_i$ is a convex, closed, and proper function for all $i\in\mcf{N}$.
In order to update its decision, each node $i\in\mcf{N}$ shares its own decision vector with its neighbors and uses only its own objective function.
Both methods are referred to as \emph{consensus} \gls{ADMM} and are outlined in \Alg~\ref{alg:c-admm} and \Alg~\ref{alg:c-admm-sum} in \App~\ref{app:consensus-admm}.

The structure of our problem \ref{eqn:main} is not suitable for applying the consensus \gls{ADMM} directly since it cannot be cast in the form of problem \eqref{eqn:primal}.
However, as we will show in the sequel, the dual of \ref{eqn:main} has the same structure as \eqref{eqn:primal}.
A similar approach was used in \cite{Chang:2015} for solving a subclass of \ref{eqn:main} in which $\mcf{K}=\{0\}$.

To this end, we rewrite \ref{eqn:main} as
\begin{equation*}
	\MinProblem{(x,w)}{\displaystyle \sum_{i\in\mcf{N}} f_i(x_i) + \mcf{I}_\mcf{K}(w)}{\displaystyle \sum_{i\in\mcf{N}} (A_i x_i - b_i) = w,}
\end{equation*}
then form its Lagrangian,
\ifpreprint
	\begin{equation}\label{eqn:main:Lagrangian}
		\mcf{L}( x, w, y) \eqdef \sum_{i\in\mcf{N}} f_i(x_i) + \mcf{I}_\mcf{K}(w) + \Big\langle y, \, \sum_{i\in\mcf{N}} (A_i x_i - b_i) - w \Big\rangle,
	\end{equation}
\fi
and derive the dual function,
\ifpreprint
	\begin{align*}
		g(y) &\eqdef \inf_{(x,w)} \mcf{L}(x,w,y) \\
		&= \inf_x \Big\lbrace \sum_{i\in\mcf{N}} \left( f_i(x_i) + \innerprod{y}{A_i x_i} \right) \Big\rbrace + \inf_w \left\lbrace \mcf{I}_\mcf{K}(w) - \innerprod{y}{w} \right\rbrace - \sum_{i\in\mcf{N}} \innerprod{y}{b_i} \\
		&= -\sum_{i\in\mcf{N}} \sup_{x_i} \big\lbrace \innerprod{x_i}{-A_i^* y} - f_i(x_i) \big\rbrace - \sup_{w\in\mcf{K}} \innerprod{y}{w} - \sum_{i\in\mcf{N}} \innerprod{y}{b_i} \\
		&= -\sum_{i\in\mcf{N}} f_i^*(-A_i^* y) - \mcf{I}_{\polar{\mcf{K}}}(y) - \sum_{i\in\mcf{N}} \innerprod{y}{b_i}.
	\end{align*}
\fi
The dual problem is then to maximize the dual function, \ie
\begin{equation}\label{eqn:dual}
	\underset{y}{\textrm{maximize}} \quad -\sum_{i\in\mcf{N}} \big( f_i^*(-A_i^* y) + \innerprod{y}{b_i} + \mcf{I}_{\polar{\mcf{K}}}(y) \big),
	\tag{$\mcf{D}$}
\end{equation}
where we used the fact that $\mcf{I}_{\polar{\mcf{K}}} = \abs{\mcf{N}} \, \mcf{I}_{\polar{\mcf{K}}}$.
Due to Assumption~\ref{ass:problem}, the optimal values of \ref{eqn:main} and \ref{eqn:dual} are finite and equal, and thus the function $\left( f_i^*\circ(-A_i^*)+\mcf{I}_{\polar{\mcf{K}}} \right)$ is proper for all $i\in\mcf{N}$.
This property of the objective functions will be used in the derivation of the algorithms.

We can now apply the consensus \gls{ADMM} for solving the dual problem.
We present in the sequel two variants based on \Alg~\ref{alg:c-admm} and \Alg~\ref{alg:c-admm-sum}.

\subsection{Aggregate variant}
The first method for solving \ref{eqn:main} is obtained by applying \Alg~\ref{alg:c-admm} to \ref{eqn:dual} where
\[
	\psi_i(y) = f_i^*(-A_i^* y) + \innerprod{y}{b_i} + \mcf{I}_{\polar{\mcf{K}}}(y).
\]
In step~\ref{alg:c-admm:psi} of \Alg~\ref{alg:c-admm} each agent needs to solve the following subproblem:
\ifpreprint
	\begin{align*}
		&\phantom{{}=} \underset{y_i}{\text{min}} \Big\lbrace f_i^*(-A_i^* y_i) + \mcf{I}_{\polar{\mcf{K}}}(y_i) + \innerprod{y_i}{b_i + p_i^{k+1}} + \rho \sum_{j\in\mcf{N}_i} \norm{y_i - \tfrac{y_i^k + y_j^k}{2}}^2 \Big\rbrace \\
		&\!= \underset{y_i}{\text{min}} \Big\lbrace f_i^*(-A_i^* y_i) + \mcf{I}_{\polar{\mcf{K}}}(y_i) + \rho d_i \norm{y_i - \tfrac{1}{2\rho d_i} r_i^{k+1} }^2 \Big\rbrace,
	\end{align*}
\fi
where
\[
	r_i^{k+1} \eqdef \rho \sum_{j\in\mcf{N}_i} (y_i^k+y_j^k) - (b_i + p_i^{k+1}).
\]
Due to Lemma~\ref{lem:prox-conjugate} (in \App~\ref{app:supporting-results}), the solution to the optimization problem above can be characterized as
\[
	y_i^{k+1} = \tfrac{1}{2\rho d_i} \project{\polar{\mcf{K}}}(A_i x_i^{k+1} + r_i^{k+1}),
\]
where
\ifpreprint
	\[
		(x_i^{k+1},t_i^{k+1}) \in \argmin_{(x_i,t_i)} \left\lbrace f_i(x_i) + \mcf{I}_\mcf{K}(t_i) + \tfrac{1}{4\rho d_i} \norm{A_i x_i + r_i^{k+1} - t_i}^2 \right\rbrace.
	\]
\fi
Notice that, if the projection onto $\mcf{K}$ can be evaluated efficiently, then the same holds for its polar cone.
Indeed, due to the Moreau decomposition \cite[\Thm~6.30]{Bauschke:2011}, we have
\[
	\project{\polar{\mcf{K}}}(z) = z - \project{\mcf{K}}(z).
\]

The proposed method is summarized in \Alg~\ref{alg:adc-admm} and can be seen as an extension of \cite[\Alg~3]{Chang:2015} since the two algorithms coincide when $\mcf{K}=\{0\}$.
Note that in this case steps~\ref{alg:adc-admm:x} and~\ref{alg:adc-admm:y} of \Alg~\ref{alg:adc-admm} reduce to
\begin{align*}
	x_i^{k+1} &\gets \argmin_{x_i} \left\lbrace f_i(x_i) + \tfrac{1}{4\rho d_i} \norm{A_i x_i + r_i^{k+1}}^2 \right\rbrace \\
	y_i^{k+1} &\gets \tfrac{1}{2\rho d_i} (A_i x_i^{k+1} + r_i^{k+1}).
\end{align*}

\begin{algorithm}[t]
	\caption{Aggregate dual consensus \gls{ADMM} for \ref{eqn:main}.}
	\label{alg:adc-admm}
	\begin{algorithmic}[1]
		\State \textbf{given} parameter $\rho>0$ and initial value $y_i^0$ for each node $i\in\mcf{N}$
		\State Set $k=0$ and $p_i^0=0$
		\Repeat
			\State Exchange $y_i^k$ with nodes in $\mcf{N}_i$
			\State $p_i^{k+1} \gets p_i^k + \rho \sum_{j\in\mcf{N}_i} (y_i^k - y_j^k)$
			\State $r_i^{k+1} \gets \rho \sum_{j\in\mcf{N}_i} (y_i^k+y_j^k) - (b_i + p_i^{k+1})$
			\ifpreprint
				\State $(x_i^{k+1},t_i^{k+1}) \gets \argmin\limits_{(x_i,t_i)} \left\lbrace f_i(x_i) + \mcf{I}_\mcf{K}(t_i) + \tfrac{1}{4 \rho d_i} \norm{ A_i x_i + r_i^{k+1} - t_i}^2 \right\rbrace$ \label{alg:adc-admm:x}
			\fi
			\State $y_i^{k+1} \gets \tfrac{1}{2 \rho d_i} \project{\polar{\mcf{K}}} \left( A_i x_i^{k+1} + r_i^{k+1} \right)$ \label{alg:adc-admm:y}
			\State $k \gets k+1$
		\Until{termination condition is satisfied}
	\end{algorithmic}
\end{algorithm}

Even though the algorithm solves the dual of \ref{eqn:main}, it also generates a primal solution to the problem, as stated in the following proposition which we prove
\ifpreprint
	in \App~\ref{app:prop:adc-admm-cvg}.
\fi

\begin{proposition}\label{prop:adc-admm-cvg}
	For all $i\in\mcf{N}$ the sequence $\seq{y_i^k}$ generated by \Alg~\ref{alg:adc-admm} converges to $y^\star$ which is a maximizer of~\ref{eqn:dual}.
	Moreover, any limit point of the sequence $\seq{(x_i^k)_{i\in\mcf{N}}}$ is a minimizer of \ref{eqn:main}.
\end{proposition}

\subsection{Decomposed variant}
In some cases the conic constraint in step~\ref{alg:adc-admm:x} of \Alg~\ref{alg:adc-admm} makes the subproblem hard to solve.
We therefore propose another method for solving \ref{eqn:main} which is obtained by applying \Alg~\ref{alg:c-admm-sum} to \ref{eqn:dual} with
\begin{equation*}
	\varphi_i(y) = f_i^*(-A_i^* y) + \innerprod{y}{b_i}
	\quad\text{and}\quad
	\vartheta_i(y) = \mcf{I}_{\polar{\mcf{K}}}(y).
\end{equation*}

In step~\ref{alg:c-admm-sum:phi} of \Alg~\ref{alg:c-admm-sum} each agent solves the following subproblem:
\ifpreprint
	\begin{align*}
		&\phantom{{}=} \underset{y_i}{\text{min}} \Big\lbrace f_i^*(-A_i^* y_i) + \innerprod{y_i}{b_i + p_i^{k+1} + s_i^{k+1}} + \tfrac{\sigma}{2}\norm{y_i - z_i^k}^2 + \rho \sum_{j\in\mcf{N}_i} \norm{y_i - \tfrac{y_i^k + y_j^k}{2}}^2 \Big\rbrace \\
		&\!= \underset{y_i}{\text{min}} \Big\lbrace f_i^*(-A_i^* y_i) + \tfrac{\sigma+2\rho d_i}{2} \norm{y_i - \tfrac{1}{\sigma+2\rho d_i} r_i^{k+1} }^2 \Big\rbrace,
	\end{align*}
\fi
where
\[
	r_i^{k+1} \eqdef \sigma z_i^k + \rho \sum_{j\in\mcf{N}_i} (y_i^k+y_j^k) - (b_i + p_i^{k+1} + s_i^{k+1}).
\]
Due to Lemma~\ref{lem:prox-conjugate}, the solution to the problem above can be characterized as
\[
	y_i^{k+1} = \tfrac{1}{\sigma+2\rho d_i} (A_i x_i^{k+1} + r_i^{k+1}),
\]
where
\[
	x_i^{k+1} \in \argmin_{x_i} \left\lbrace f_i(x_i) + \tfrac{1}{2(\sigma+2\rho d_i)} \norm{A_i x_i + r_i^{k+1}}^2 \right\rbrace.
\]

\begin{algorithm}[t]
	\caption{Decomposed dual consensus \gls{ADMM} for \ref{eqn:main}.}
	\label{alg:ddc-admm}
	\begin{algorithmic}[1]
		\State \textbf{given} parameters $\sigma>0$, $\rho>0$ and initial values $y_i^0,z_i^0,s_i^0$ for each node $i\in\mcf{N}$
		\State Set $k=0$ and $p_i^0=0$
		\Repeat
			\State Exchange $y_i^k$ with nodes in $\mcf{N}_i$
			\State $p_i^{k+1} \gets p_i^k + \rho \sum_{j\in\mcf{N}_i} (y_i^k - y_j^k)$
			\State $s_i^{k+1} \gets s_i^k + \sigma (y_i^k - z_i^k)$
			\State $r_i^{k+1} \gets \sigma z_i^k + \rho \sum_{j\in\mcf{N}_i} (y_i^k+y_j^k) - (b_i + p_i^{k+1} + s_i^{k+1})$
			\State $x_i^{k+1} \gets \argmin\limits_{x_i} \Big\lbrace f_i(x_i) + \tfrac{1}{2(\sigma + 2 \rho d_i)} \norm{ A_i x_i + r_i^{k+1} }^2 \Big\rbrace$ \label{alg:ddc-admm:x}
			\State $y_i^{k+1} \gets \tfrac{1}{\sigma + 2 \rho d_i} \left( A_i x_i^{k+1} + r_i^{k+1} \right)$ \label{alg:ddc-admm:y}
			\State $z_i^{k+1} \gets \project{\polar{\mcf{K}}} \left( y_i^{k+1} + \tfrac{1}{\sigma} s_i^{k+1} \right)$ \label{alg:ddc-admm:z}
			\State $k \gets k+1$
		\Until{termination condition is satisfied}
	\end{algorithmic}
\end{algorithm}

The proposed algorithm is summarized in \Alg~\ref{alg:ddc-admm}.
The following proposition,
\ifpreprint
	proven in \App~\ref{app:prop:ddc-admm-cvg},
\fi
states the convergence result.

\begin{proposition}\label{prop:ddc-admm-cvg}
	For all $i\in\mcf{N}$ the sequences $\seq{y_i^k}$ and $\seq{z_i^k}$ generated by \Alg~\ref{alg:ddc-admm} converge to the same vector $y^\star$ which is a maximizer of \ref{eqn:dual}.
	Moreover, any limit point of the sequence $\seq{(x_i^k)_{i\in\mcf{N}}}$ is a minimizer of \ref{eqn:main}.
\end{proposition}

In both proposed methods each agent communicates only with its neighbors and requires no global information about the graph.
Also, the agents can update their decision vectors in parallel since they only use their neighbors' information from the previous iteration.
Finally, the methods converge for any positive values of their parameters, making them robust against noisy and unreliable problem data.
Although \Alg~\ref{alg:ddc-admm} has simpler iterations than \Alg~\ref{alg:adc-admm}, it is expected to converge slower due to additional regularization terms in the augmented Lagrangian associated with the method;
\ifpreprint
	see \App~A for more details.
\fi

Observe that in both algorithms each agent solves a sequence of optimization problems parameterized in $r_i^{k+1}$.
Provided that optimization solvers used by the agents can be warm-started (see \eg \cite{Ferreau:2014,Stellato:2018,Garstka:2019}), the computational burden of the proposed algorithms can be reduced significantly.

\begin{remark}
	The objective function in step~\ref{alg:adc-admm:x} of \Alg~\ref{alg:adc-admm} is not necessarily strongly convex, and thus the set of minimizers is not a singleton in general.
	However, as the function $\left( f_i^*\circ(-A_i^*)+\mcf{I}_{\polar{\mcf{K}}} \right)$ is proper, the optimization problem in \Alg~\ref{alg:adc-admm} has at least one solution due to Lemma~\ref{lem:prox-conjugate}.
	The same holds for the optimization problem in step~\ref{alg:ddc-admm:x} of \Alg~\ref{alg:ddc-admm}.
\end{remark}

\section{Numerical example}

Consider the \glsentrylong{BPD} problem:
\begin{equation}
	\MinProblem[{eqn:bpd}]{u}{\norm{u}_1}{\norm{Ru-r}_2 \le \eps,}
\end{equation}
with decision variable $u\in\Re^q$ and problem data $R\in\Re^{p \times q}$, $r\in\Re^p$, and $\eps \ge 0$.
The problem arises in compressed sensing where the goal is to recover a sparse vector $u$ from noisy measurements $r \approx Ru$ \cite{Donoho:2006}.

The dimensions of \eqref{eqn:bpd} can be very large, making it challenging to solve on a single computational unit.
To solve the problem in a distributed fashion, we partition $u$ into $\abs{\mcf{N}}$ blocks so that $u=(u_i)_{i\in\mcf{N}}$.
We then interpret each of these blocks as nodes and connect them through a communication graph $\mcf{G}=(\mcf{N},\mcf{E})$.
The resulting problem is
\begin{equation*}
	\MinProblem{u}{\displaystyle \sum_{i\in\mcf{N}} \norm{u_i}_1}{\normbig{\sum_{i\in\mcf{N}} (R_i u_i-r_i)}_2 \le \eps,}
\end{equation*}
where $R=[R_i]_{i\in\mcf{N}}$, and $r_i=r/\abs{\mcf{N}}$ for all $i\in\mcf{N}$.
Note that the problem above can be reformulated in the form of \ref{eqn:main}, \ie
\begin{equation*}
	\MinProblem{(u,v)}{\displaystyle \sum_{i\in\mcf{N}} \big( \norm{u_i}_1 + \mcf{I}_{ \{ \eps/\abs{\mcf{N}} \} }(v_i) \big)}{\displaystyle \sum_{i\in\mcf{N}} ( R_i u_i-r_i, v_i ) \in \mcf{S},}
\end{equation*}
where $v\eqdef(v_i)_{i\in\mcf{N}}$, $v_i\in\Re$, and $\mcf{S} \eqdef \{ (z,t)\in\Re^p\times\Re \mid \norm{z}_2 \le t \}$ is the second-order cone whose projection can be evaluated in a closed form \cite[\S 6.3.2]{Parikh:2014}.

We generate the problem data as described in \cite{Aybat:2016:arxiv}, \ie we set $p=20$, $q=120$, each element of $R$ is \iid drawn from the standard normal distribution, $r = R u^\star + \eta$ where $u^\star$ is generated by choosing $\kappa=20$ of its elements, uniformly at random, drawn from the standard normal distribution, and the rest of the elements are set to zero, while $\eta$ is a noise vector whose elements are \iid drawn from $\mcf{N}(0, \kappa\, 10^{-4})$, and $\eps>0$ is chosen so that the probability that $\norm{\eta}_2 \le \eps$ is equal to $0.95$.
Finally, we generate $\mcf{G}$ as a random small-world network with $\abs{\mcf{N}}=10$ nodes and $\abs{\mcf{E}}=15$ edges so that $\abs{\mcf{N}}$ edges create a random cycle over nodes, and the remaining $\abs{\mcf{E}}-\abs{\mcf{N}}$ edges are selected uniformly at random.
We partition $u$ into $\abs{\mcf{N}}$ blocks of the same dimensions, so that $R_i\in\Re^{p\times(q/\abs{\mcf{N}})}$ for all $i\in\mcf{N}$.

\ifpreprint
	\begin{figure*}[t]
		\centering
		\footnotesize
		\begin{tikzpicture}
			\begin{axis}[%
				hide axis,
				xmin=10,	
				xmax=50,	
				ymin=0,		
				ymax=0.4,	
				legend cell align=left,
				legend columns=3,
				legend style={nodes={scale=1},
							  fill=white,
							  fill opacity=1,
							  draw opacity=1,
							  text opacity=1,
							  /tikz/every even column/.append style={column sep=3em}},
				every axis plot/.append style={very thick}
				]
				\addlegendimage{blue}
				\addlegendentry{\Alg~\ref{alg:adc-admm}};
				\addlegendimage{red}
				\addlegendentry{\Alg~\ref{alg:ddc-admm}};
				\addlegendimage{green}
				\addlegendentry{DPDA};
			\end{axis}
		\end{tikzpicture}
		\\[.5em]
		\begin{tabular}{rl}
			\begin{tikzpicture}[baseline,trim axis left]
				\begin{axis}[
					width=.49\textwidth,
					height=.4\textwidth,
					ymin = 4e-7,
					ymax = 1e1,
					xmin = -200,
					xmax = 2200,
					title={$\displaystyle \frac{\abs{\norm{u^k}_1-\norm{u^\star}_1}}{\norm{u^\star}_1}$},
					title style={at={(.9,.8)}, anchor=east},
					ylabel={Relative suboptimality},
					ytick={1e-6,1e-4,1e-2,1e0},
					yticklabel pos=left,
					ymode=log,
					xticklabels={,,},
					every axis plot/.append style={very thick}
					]
					\addplot [blue, select coords between index={1}{2000}] table[x=iter, y=relSubopt, col sep=comma] {data/mean_stats_ADCADMM.csv};
					\addplot [red, select coords between index={1}{2000}] table[x=iter, y=relSubopt, col sep=comma] {data/mean_stats_DDCADMM.csv};
					\addplot [green, select coords between index={1}{2000}] table[x=iter, y=relSubopt, col sep=comma] {data/mean_stats_DPDA.csv};
				\end{axis}
			\end{tikzpicture}
			&
			\begin{tikzpicture}[baseline,trim axis right]
				\begin{axis}[
					width=.49\textwidth,
					height=.4\textwidth,
					ymin = 1e-6,
					ymax = 1e2,
					xmin = -200,
					xmax = 2200,
					title={$\max\left( \norm{Ru^k-r}_2-\eps, \, 0 \right)$},
					title style={at={(.9,.8)}, anchor=east},
					ylabel={Infeasibility},
					yticklabel pos=right,
					ymode=log,
					xticklabels={,,},
					every axis plot/.append style={very thick}
					]
					\addplot [blue, select coords between index={1}{2000}] table[x=iter, y=infeas, col sep=comma] {data/mean_stats_ADCADMM.csv};
					\addplot [red, select coords between index={1}{2000}] table[x=iter, y=infeas, col sep=comma] {data/mean_stats_DDCADMM.csv};
					\addplot [green, select coords between index={1}{2000}] table[x=iter, y=infeas, col sep=comma] {data/mean_stats_DPDA.csv};
				\end{axis}
			\end{tikzpicture}
			\\[-.75em]
			\begin{tikzpicture}[baseline,trim axis left]
				\begin{axis}[
					width=.49\textwidth,
					height=.4\textwidth,
					ymin = 6e-6,
					ymax = 1e0,
					xmin = -200,
					xmax = 2200,
					title={$\norm{u^k-u^\star}_2$},
					title style={at={(.9,.8)}, anchor=east},
					xlabel={Iteration $k$},
					ylabel={Distance to a solution},
					yticklabel pos=left,
					ymode=log,
					every axis plot/.append style={very thick}
					]
					\addplot [blue, select coords between index={1}{2000}] table[x=iter, y=consViol, col sep=comma] {data/mean_stats_ADCADMM.csv};
					\addplot [red, select coords between index={1}{2000}] table[x=iter, y=consViol, col sep=comma] {data/mean_stats_DDCADMM.csv};
					\addplot [green, select coords between index={1}{2000}] table[x=iter, y=consViol, col sep=comma] {data/mean_stats_DPDA.csv};
				\end{axis}
			\end{tikzpicture}%
			&
			\begin{tikzpicture}[baseline,trim axis right]
				\begin{axis}[
					width=.49\textwidth,
					height=.4\textwidth,
					ymin=3e-4,
					ymax=1e1,
					xmin = -200,
					xmax = 2200,
					title={$\max\limits_{i\in\mcf{N}} \, \norm{y_i^k-\bar{y}^k}_2$},
					title style={at={(.9,.8)}, anchor=east},
					xlabel={Iteration $k$},
					ylabel={Consensus violation},
					yticklabel pos=right,
					ymode=log,
					every axis plot/.append style={very thick}
					]
					\addplot [blue, select coords between index={1}{2000}] table[x=iter, y=solDist, col sep=comma] {data/mean_stats_ADCADMM.csv};
					\addplot [red, select coords between index={1}{2000}] table[x=iter, y=solDist, col sep=comma] {data/mean_stats_DDCADMM.csv};
					\addplot [green, select coords between index={1}{2000}] table[x=iter, y=solDist, col sep=comma] {data/mean_stats_DPDA.csv};
				\end{axis}
			\end{tikzpicture}
		\end{tabular}
		\caption{%
			Numerical performance of \Alg~\ref{alg:adc-admm}, \Alg~\ref{alg:ddc-admm} and \gls{DPDA} \cite{Aybat:2016,Aybat:2016:arxiv} for solving the basis pursuit denoising problem \eqref{eqn:bpd}, where $u^\star$ denotes its optimal solution, and $\bar{y}^k \eqdef \tfrac{1}{\abs{\mcf{N}}} \sum_{i\in\mcf{N}} y_i^k$.
			We show the mean results over $10$ randomly generated instances of the problem.
		}
		\label{fig:numerical-comparison}
	\end{figure*}
\fi

We compare our methods to \gls{DPDA} \cite{Aybat:2016,Aybat:2016:arxiv} which is a decentralized and parallelizable algorithm that has recently been proposed for solving \ref{eqn:main}.
\Fig~\ref{fig:numerical-comparison} shows numerical performance of \Alg~\ref{alg:adc-admm}, \Alg~\ref{alg:ddc-admm} and \gls{DPDA} for solving \eqref{eqn:bpd}.
As performance metrics, we consider the mean values of relative suboptimality, infeasibility, distance to a solution, and consensus violation over $10$ different problem instances.
For each of these instances we randomly generate both the network and the problem data.
The parameters of \gls{DPDA} are chosen depending on the problem data as suggested in \cite{Aybat:2016:arxiv}, while the parameters appearing in \Alg~\ref{alg:adc-admm} and \Alg~\ref{alg:ddc-admm} are set to $\sigma=\rho=1$.

It can be seen that \Alg~\ref{alg:adc-admm} and \Alg~\ref{alg:ddc-admm} require a smaller number of iterations than \gls{DPDA} for attaining the same accuracy.
However, the computational complexity of performing each iteration of \gls{DPDA} is lower.
More specifically, \gls{DPDA} only evaluates the proximal operator of the $\ell_1$--norm, which has a closed-form solution \cite[\S 6.5.2]{Parikh:2014}.
In contrast, in each iteration \Alg~\ref{alg:adc-admm} and \Alg~\ref{alg:ddc-admm} solve a \glsentrylong{SOCP} and a \glsentrylong{QP}, respectively.
This means that \Alg~\ref{alg:adc-admm} and \Alg~\ref{alg:ddc-admm} are preferred over \gls{DPDA} when the cost of agent-to-agent communication outweighs the cost of computations performed by the agents.
Since the convergence rates of \Alg~\ref{alg:adc-admm} and \Alg~\ref{alg:ddc-admm} with respect to the number of iterations are very similar, the latter method is more efficient due to simpler optimization problems solved by the agents.

Note that the time complexity of the algorithms depends not only on the iteration complexity, but also on communication delays and properties of optimization solvers used by the agents such as precision, whether they support warm-starting etc.

\section{Conclusion}
We propose two methods based on \gls{ADMM} for solving resource allocation problems over a network of computational agents.
Both methods are fully parallelizable and decentralized in the sense that each agent exchanges information only with its neighbors in the network and requires only its own data for updating its decision.
We prove convergence of both methods for any positive values of the algorithm parameters.
Our methods are compared numerically against a competing method, and were shown to require a smaller number of iterations to attain the same accuracy.

\ifpreprint
	\appendix
\fi
\setcounter{algorithm}{0}
\renewcommand{\thealgorithm}{\thesection.\arabic{algorithm}}

\section{Consensus ADMM}\label{app:consensus-admm}

The authors in \cite{Mateos:2010} propose two decentralized methods for solving the consensus optimization problem \eqref{eqn:primal} over a connected undirected graph.
The first method assumes that the proximal operator of $\psi_i$ can be evaluated efficiently, and is outlined in \Alg~\ref{alg:c-admm}.

The second method assumes that $\psi_i$ can be represented as the sum of two functions, \ie
\[
	\psi_i(y) = \varphi_i(y) + \vartheta_i(y),
\]
where both $\varphi_i$ and $\vartheta_i$ are convex, closed, and proper.
It is often the case that the proximal operator of $\psi_i$ is much harder to evaluate than the proximal operators of $\varphi_i$ and $\vartheta_i$.
This is the reason for introducing another method that evaluates the proximal operators of $\varphi_i$ and $\vartheta_i$ instead.
The method is outlined in \Alg~\ref{alg:c-admm-sum}.

\begin{algorithm}[t]
	\caption{Consensus \gls{ADMM} for \eqref{eqn:primal}.}
	\label{alg:c-admm}
	\begin{algorithmic}[1]
		\State \textbf{given} parameter $\rho>0$ and initial value $y_i^0$ for each node $i\in\mcf{N}$
		\State Set $k=0$ and $p_i^0=0$
		\Repeat
		\State Exchange $y_i^k$ with nodes in $\mcf{N}_i$
		\State $p_i^{k+1} \gets p_i^k + \rho \sum_{j\in\mcf{N}_i} (y_i^k - y_j^k)$
		\ifpreprint
			\State $y_i^{k+1} \gets \argmin\limits_{y_i} \left\lbrace \psi_i(y_i) + \innerprod{y_i}{p_i^{k+1}} + \rho \sum_{j\in\mcf{N}_i} \normbig{y_i - \frac{y_i^k + y_j^k}{2}}^2 \right\rbrace$ \label{alg:c-admm:psi}
		\fi
		\State $k \gets k+1$
		\Until{termination condition is satisfied}
	\end{algorithmic}
\end{algorithm}

\begin{algorithm}[t]
	\caption{Consensus \gls{ADMM} for \eqref{eqn:primal} where $\psi_i=\varphi_i+\vartheta_i$.}
	\label{alg:c-admm-sum}
	\begin{algorithmic}[1]
		\State \textbf{given} parameters $\sigma>0$, $\rho>0$ and initial values $y_i^0,z_i^0,s_i^0$ for each node $i\in\mcf{N}$
		\State Set $k=0$ and $p_i^0=0$
		\Repeat
		\State Exchange $y_i^k$ with nodes in $\mcf{N}_i$
		\State $p_i^{k+1} \gets p_i^k + \rho \sum_{j\in\mcf{N}_i} (y_i^k - y_j^k)$
		\State $s_i^{k+1} \gets s_i^k + \sigma (y_i^k - z_i^k)$
		\ifpreprint
			\State $y_i^{k+1} \gets \argmin\limits_{y_i} \left\lbrace \varphi_i(y_i) + \innerprod{y_i}{p_i^{k+1} + s_i^{k+1}} + \frac{\sigma}{2} \norm{y_i - z_i^k}^2 + \rho \sum_{j\in\mcf{N}_i} \normbig{y_i - \frac{y_i^k + y_j^k}{2}}^2 \right\rbrace$ \label{alg:c-admm-sum:phi}
			\State $z_i^{k+1} \gets \argmin\limits_{z_i} \Big\lbrace \vartheta_i(z_i) - \innerprod{z_i}{s_i^{k+1}} + \frac{\sigma}{2} \norm{z_i - y_i^{k+1}}^2 \Big\rbrace$
		\fi
		\State $k \gets k+1$
		\Until{termination condition is satisfied}
	\end{algorithmic}
\end{algorithm}

Derivations of both algorithms can be found in~\cite{Mateos:2010},
\ifpreprint
	but we also include them here for the sake of completeness.
\fi

\ifpreprint

	\subsection{Derivation of \Alg~\ref{alg:c-admm}}\label{app:c-admm}

	Since $(\mcf{N},\mcf{E})$ is a connected undirected graph, \eqref{eqn:primal} can be reformulated as
	\begin{equation*}
		\MinProblem{}{\displaystyle \sum_{i\in\mcf{N}} \psi_i(y_i)}{y_i=t_{ij}, \quad i\in\mcf{N}, \: j\in\mcf{N}_i, \\ & y_j=t_{ij}, \quad i\in\mcf{N}, \: j\in\mcf{N}_i.}
	\end{equation*}
	The augmented Lagrangian associated with the problem above has the form
	\begin{align*}
		\mcf{L}_\rho( y, t, (u, v) ) \eqdef \sum_{i\in\mcf{N}} \Big[ \psi_i(y_i) + \sum_{j\in\mcf{N}_i} \big( & \innerprod{u_{ij}}{y_i-t_{ij}}  + \tfrac{\rho}{2} \norm{y_i-t_{ij}}^2 + \\[-.75em]
		& \innerprod{v_{ij}}{y_j-t_{ij}} + \tfrac{\rho}{2} \norm{y_j-t_{ij}}^2 \big) \Big].
	\end{align*}
	\gls{ADMM} then consists of the following iterations \cite{Boyd:2011}:
	\begin{align}
		y_i^{k+1} &\gets \argmin_{y_i} \Big\lbrace \psi_i(y_i) + \sum_{j\in\mcf{N}_i} \left( \innerprod{y_i}{u_{ij}^k + v_{ji}^k} + \tfrac{\rho}{2} \norm{y_i - t_{ij}^k}^2 + \tfrac{\rho}{2} \norm{y_i - t_{ji}^k}^2 \right) \Big\rbrace \label{eqn:admm:yi} \\
		t_{ij}^{k+1} &\gets \argmin_{t_{ij}} \Big\lbrace -\innerprod{t_{ij}}{u_{ij}^k + v_{ij}^k} + \tfrac{\rho}{2} \norm{t_{ij}-y_i^{k+1}}^2 + \tfrac{\rho}{2} \norm{t_{ij}-y_j^{k+1}}^2 \Big\rbrace  \label{eqn:admm:tij} \\
		u_{ij}^{k+1} &\gets u_{ij}^k + \rho \left( y_i^{k+1} - t_{ij}^{k+1} \right) \label{eqn:admm:uij} \\
		v_{ij}^{k+1} &\gets v_{ij}^k + \rho \left( y_j^{k+1} - t_{ij}^{k+1} \right). \label{eqn:admm:vij}
	\end{align}
	The minimization problem in \eqref{eqn:admm:tij} has the following closed-form solution:
	\[
		t_{ij}^{k+1} = \half \left( y_i^{k+1} + y_j^{k+1} \right) + \tfrac{1}{2\rho} \left( u_{ij}^{k} + v_{ij}^{k} \right).
	\]
	Summing \eqref{eqn:admm:uij} and \eqref{eqn:admm:vij}, and plugging $t_{ij}^{k+1}$ from the equality above, we obtain
	\begin{equation}\label{eqn:app:uij_vij}
		u_{ij}^{k+1} + v_{ij}^{k+1} = 0,
	\end{equation}
	which then implies
	\begin{equation}\label{eqn:app:tij-update}
		t_{ij}^{k+1} = \half \left( y_i^{k+1} + y_j^{k+1} \right),
	\end{equation}
	and
	\begin{equation}\label{eqn:app:uij-update}
		u_{ij}^{k+1} = u_{ij}^k + \tfrac{\rho}{2} \left( y_i^{k+1} - y_j^{k+1} \right).
	\end{equation}
	Note from \eqref{eqn:app:tij-update} that if $t_{ij}^0=t_{ji}^0$, then $t_{ij}^k=t_{ji}^k$ for all $k\in\Nat$.
	Also, it follows from \eqref{eqn:app:uij_vij} and \eqref{eqn:app:uij-update} that if $u_{ij}^0 = v_{ij}^0 = 0$ and $u_{ij}^0 = u_{ji}^0 = 0$, then $u_{ij}^k = -v_{ij}^k$ and $u_{ij}^k = -u_{ji}^k$ for all $k\in\Nat$.
	Defining
	\[
		p_i^k \eqdef \sum_{j\in\mcf{N}_i} \left( u_{ij}^k - v_{ij}^k \right) = 2\sum_{j\in\mcf{N}_i} u_{ij}^k,
	\]
	we have
	\begin{equation}\label{eqn:app:pi}
		p_i^{k+1} \eqdef p_i^k + \rho \sum_{j\in\mcf{N}_i} \left( y_i^{k+1} - y_j^{k+1} \right).
	\end{equation}
	Finally, iterations \eqref{eqn:admm:yi}--\eqref{eqn:admm:vij} reduce to
	\begin{align*}
		y_i^{k+1} &\gets \argmin_{y_i} \Big\lbrace \psi_i(y_i) + \innerprod{y_i}{p_i^k} + \rho \sum_{j\in\mcf{N}_i} \norm{y_i - \tfrac{y_i^k+y_j^k}{2}}^2 \Big\rbrace \\
		p_i^{k+1} &\gets p_i^k + \rho \sum_{j\in\mcf{N}_i} \left( y_i^{k+1} - y_j^{k+1} \right).
	\end{align*}
	\Alg~\ref{alg:c-admm} is obtained by starting the iteration from the $p_i$-update.
	Note that summing \eqref{eqn:app:pi} over $i\in\mcf{N}$, we obtain
	\begin{equation}\label{eqn:app:pi-sum}
		\sum_{i\in\mcf{N}} p_i^{k+1} = \sum_{i\in\mcf{N}} p_i^k + \rho \sum_{i\in\mcf{N}} \sum_{j\in\mcf{N}_i} \left( y_i^{k+1} - y_j^{k+1} \right) = 0,
	\end{equation}
	where the second equality follows from $p_i^0=0$ and the symmetry in the double sum.

	\subsection{Derivation of \Alg~\ref{alg:c-admm-sum}}

	Problem \eqref{eqn:primal} in which $\psi_i = \varphi_i + \vartheta_i$ can be reformulated as
	\begin{equation*}
		\MinProblem{}{\displaystyle \sum_{i\in\mcf{N}} \big( \varphi_i(y_i) + \vartheta_i(z_i) \big)}{y_i=z_i, \quad \; i\in\mcf{N}, \\ & y_i=t_{ij}, \quad i\in\mcf{N}, \: j\in\mcf{N}_i, \\ & y_j=t_{ij}, \quad i\in\mcf{N}, \: j\in\mcf{N}_i.}
	\end{equation*}
	The augmented Lagrangian associated with the problem above has the form
	\begin{align*}
		\mcf{L}_{\sigma,\rho}( y, (z, t), (s, u, v) ) \eqdef \sum_{i\in\mcf{N}} \Big[ & \varphi_i(y_i) + \vartheta_i(z_i) + \innerprod{s_i}{y_i-z_i} + \tfrac{\sigma}{2} \norm{y_i-z_i}^2 + \\
		& \sum_{j\in\mcf{N}_i} \big( \innerprod{u_{ij}}{y_i-t_{ij}}  + \tfrac{\rho}{2} \norm{y_i-t_{ij}}^2 + \\
		&\phantom{{} \sum_{j\in\mcf{N}_i} \big( } \innerprod{v_{ij}}{y_j-t_{ij}} + \tfrac{\rho}{2} \norm{y_j-t_{ij}}^2 \big) \Big].
	\end{align*}
	\gls{ADMM} then consists of the following iterations:
	\begin{align*}
		y_i^{k+1} &\gets \argmin_{y_i} \Big\lbrace \varphi_i(y_i) + \innerprod{y_i}{s_i^k} + \tfrac{\sigma}{2} \norm{y_i-z_i^k}^2 \\
		&\phantom{\gets \argmin_{y_i} \Big\lbrace} + \sum_{j\in\mcf{N}_i} \left( \innerprod{y_i}{u_{ij}^k + v_{ji}^k} + \tfrac{\rho}{2} \norm{y_i - t_{ij}^k}^2 + \tfrac{\rho}{2} \norm{y_i - t_{ji}^k}^2 \right) \Big\rbrace \\
		z_i^{k+1} &\gets \argmin_{z_i} \Big\lbrace \vartheta_i(z_i) - \innerprod{z_i}{s_i^k} + \tfrac{\sigma}{2} \norm{z_i-y_i^{k+1}}^2 \Big\rbrace \\
		t_{ij}^{k+1} &\gets \argmin_{t_{ij}} \Big\lbrace -\innerprod{t_{ij}}{u_{ij}^k + v_{ij}^k} + \tfrac{\rho}{2} \norm{t_{ij}-y_i^{k+1}}^2 + \tfrac{\rho}{2} \norm{t_{ij}-y_j^{k+1}}^2 \Big\rbrace \\
		s_i^{k+1} &\gets s_i^k + \sigma \left( y_i^{k+1} - z_i^{k+1} \right) \\
		u_{ij}^{k+1} &\gets u_{ij}^k + \rho \left( y_i^{k+1} - t_{ij}^{k+1} \right) \\
		v_{ij}^{k+1} &\gets v_{ij}^k + \rho \left( y_j^{k+1} - t_{ij}^{k+1} \right).
	\end{align*}
	We can eliminate $t_{ij}$ and introduce a variable $p_i$ in a similar fashion as in Section~\ref{app:c-admm}.
	Iterations above then reduce to
	\begin{align*}
		y_i^{k+1} &\gets \argmin_{y_i} \Big\lbrace \varphi_i(y_i) + \innerprod{y_i}{s_i^k+p_i^k} + \tfrac{\sigma}{2} \norm{y_i-z_i^k}^2 + \rho \sum_{j\in\mcf{N}_i} \norm{y_i - \tfrac{y_i^k+y_j^k}{2}}^2 \Big\rbrace \\
		z_i^{k+1} &\gets \argmin_{z_i} \Big\lbrace \vartheta_i(z_i) - \innerprod{z_i}{s_i^k} + \tfrac{\sigma}{2} \norm{z_i-y_i^{k+1}}^2 \Big\rbrace \\
		s_i^{k+1} &\gets s_i^k + \sigma \left( y_i^{k+1} - z_i^{k+1} \right) \\
		p_i^{k+1} &\gets p_i^k + \rho \sum_{j\in\mcf{N}_i} \left( y_i^{k+1} - y_j^{k+1} \right).
	\end{align*}
	\Alg~\ref{alg:c-admm-sum} is obtained by replacing the order of $s_i$- and $p_i$-updates, and starting the iteration from the $p_i$-update.

\fi

\section{Supporting results}\label{app:supporting-results}

\begin{lemma}\label{lem:prox-conjugate}
	Let $g:\Re^n\mapsto\ReExt$ be a convex, closed, and proper function, $\mcf{C}$ a nonempty, closed, and convex cone, and $E\in\Re^{m\times n}$.
	Consider the following function:
	\[
		d(y) \eqdef g^*(-E^* y) + \mcf{I}_\mcf{C}(y),
	\]
	and suppose it is proper.
	Then the proximal operator of $d$ can be computed as
	\[
		\prox_d^\gamma(z) = \tfrac{1}{\gamma} \project{\mcf{C}}(Ex^\star + \gamma z),
	\]
	where $(x^\star,t^\star)$ is a minimizer of the following problem:
	\ifpreprint
		\[
			\underset{(x,t)}{\textrm{minimize}} \quad g(x) + \mcf{I}_{\polar{\mcf{C}}}(t) + \tfrac{1}{2\gamma} \norm{Ex + \gamma z - t}^2,
		\]
	\fi
	which has at least one solution.
\end{lemma}
\begin{proof}
	From the definition of $d$, $\prox_d^\gamma(z)$ can be computed as the minimizer of the following problem:
	\begin{equation}\label{eqn:lem:min-y}
		\underset{y}{\textrm{minimize}} \quad \mcf{I}_\mcf{C}(y) + \tfrac{1}{\gamma} g^*(-E^* y) + \tfrac{1}{2}\norm{y-z}^2.
	\end{equation}
	Due to \cite[Prop.~19.5]{Bauschke:2011}, a solution to the problem above can be characterized as
	\[
		y^\star = \project{\mcf{C}}(z+Ep^\star),
	\]
	where
	\[\arraycolsep=0.5ex
	\begin{array}{ll}
		p^\star \in \underset{p}{\textrm{argmin}} & \Big\lbrace \tfrac{1}{2}\norm{z+Ep}^2 - \underset{s\in\mcf{C}}{\rm min} \left\lbrace \tfrac{1}{2}\norm{s-(z+Ep)}^2 \right\rbrace
		\ifpreprint
			+ \tfrac{1}{\gamma}g(\gamma p) \Big\rbrace \\
		\fi
		\phantom{p^\star} = \underset{p}{\textrm{argmin}} & \left\lbrace \tfrac{1}{2}\norm{z+Ep}^2 - \tfrac{1}{2}\dist{\mcf{C}}^2(z+Ep) + \tfrac{1}{\gamma}g(\gamma p) \right\rbrace \\
		\phantom{p^\star} = \underset{p}{\textrm{argmin}} & \left\lbrace \tfrac{1}{2}\dist{\polar{\mcf{C}}}^2(z+Ep) + \tfrac{1}{\gamma}g(\gamma p) \right\rbrace \\
		\phantom{p^\star} = \underset{p}{\textrm{argmin}} & \left\lbrace \tfrac{\gamma}{2}\dist{\polar{\mcf{C}}}^2(z+Ep) + g(\gamma p) \right\rbrace,
	\end{array}
	\]
	where we used the Moreau decomposition \cite[\Thm~6.30]{Bauschke:2011} in the second equality.
	Introducing the variable $x=\gamma p$, we can write
	\begin{align*}
		y^\star &= \tfrac{1}{\gamma}\project{\mcf{C}}(Ex^\star + \gamma z) \\
		x^\star &\in \underset{x}{\textrm{argmin}} \left\lbrace g(x) + \tfrac{1}{2\gamma}\dist{\polar{\mcf{C}}}^2(Ex + \gamma z) \right\rbrace.
	\end{align*}
	Note that, since the minimization in \eqref{eqn:lem:min-y} involves a strongly convex function, $y^\star$ is unique even when $x^\star$ is not.

	Finally, the minimization over $x$ can be written as
	\begin{equation*}
		\MinProblem{(x,t)}{g(x) + \tfrac{1}{2\gamma} \norm{Ex + \gamma z -t}^2}{t \in \polar{\mcf{C}}.}
	\end{equation*}
	This concludes the proof.
\end{proof}

\ifpreprint

	\begin{lemma}
		The first-order optimality conditions for \ref{eqn:main} are given by
		\begin{subequations}\label{eqn:optimal-cond}
			\begin{align}
				w &\in \mcf{K} \label{eqn:optimal-cond:w} \\
				0 &\in \partial f (x_i) + A_i^T y, \quad \forall i \in \mcf{N} \label{eqn:optimal-cond:stat} \\
				y &\in \normalCone{\mcf{K}}(w) \label{eqn:optimal-cond:ncone} \\
				0 &= \sum_{i\in\mcf{N}} (A_i x_i - b_i) - w. \label{eqn:optimal-cond:equal}
			\end{align}
		\end{subequations}
	\end{lemma}
	\begin{proof}
		A primal-dual solution to \ref{eqn:main} can be characterized via a saddle point of its Lagrangian given by \eqref{eqn:main:Lagrangian}.
		Therefore, the first-order optimality conditions can be written as \cite[\Thm~11.50]{Rockafellar:1998}
		\begin{align*}
			w &\in \mcf{K} \\
			0 &\in \partial_{x_i} \mcf{L}( x, w, y) = \partial f_i(x_i) + A_i^T y, \quad \forall i \in \mcf{N} \\
			0 &\in \partial_{w} \mcf{L}( x, w, y) = \partial \mcf{I}_\mcf{K}(w) - y \\
			0 &= \nabla_{y} \mcf{L}( x, w, y) = \sum_{i\in\mcf{N}} (A_i x_i - b_i) - w,
		\end{align*}
		where the third inclusion is equivalent to $y\in\normalCone{\mcf{K}}(w)$.
	\end{proof}

	\begin{lemma}\label{lem:normal-cone}
		Let $\mcf{C}$ be a nonempty, closed, and convex cone, and suppose that $y\in\normalCone{\mcf{C}}(t_i)$ for $i=1,\ldots,m$.
		Then $y\in\normalCone{\mcf{C}}(\sum_{i=1}^m t_i)$.
	\end{lemma}
	\begin{proof}
		We show below that the result holds for $m=2$.
		The general result then holds by induction.

		Inclusions $y\in\normalCone{\mcf{C}}(t_1)$ and $y\in\normalCone{\mcf{C}}(t_2)$ are equivalent to
		\begin{equation*}
			0 \ge \sup_{t_1'\in\mcf{C}}\innerprod{y}{t_1'-t_1}
			\quad\text{and}\quad
			0 \ge \sup_{t_2'\in\mcf{C}}\innerprod{y}{t_2'-t_2}.
		\end{equation*}
		Summing the inequalities above, we obtain
		\[
			0 \ge \sup_{\substack{t_1'\in\mcf{C} \\ t_2'\in\mcf{C}}} \innerprod{y}{(t_1'+t_2')-(t_1+t_2)}.
		\]
		Since $\mcf{C}$ is a convex cone, we have $\mcf{C}+\mcf{C}=\mcf{C}$, and the inequality reduces to
		\[
			0 \ge \sup_{t'\in\mcf{C}} \innerprod{y}{t'-(t_1+t_2)},
		\]
		or equivalently, $y\in\normalCone{\mcf{C}}(t_1+t_2)$.
	\end{proof}

	\section{Proof of \Prop~\ref{prop:adc-admm-cvg}}\label{app:prop:adc-admm-cvg}

	Since \Alg~\ref{alg:adc-admm} is a direct application of \Alg~\ref{alg:c-admm} to \ref{eqn:dual}, it follows from \cite[\Prop~2]{Mateos:2010} that
	\[
		y_i^k \to y^\star, \quad \forall i \in \mcf{N},
	\]
	where $y^\star$ is a maximizer of \ref{eqn:dual}.
	We show in the sequel that the iterates $(x_i^k)_{i\in\mcf{N}}$, $w^k \eqdef \sum_{i\in\mcf{N}} t_i^k$, and $y_i^k$ satisfy optimality conditions~\eqref{eqn:optimal-cond} in the limit.

	Since $(x_i^{k+1},t_i^{k+1})$ is a minimizer of the optimization problem in step~\ref{alg:adc-admm:x} of \Alg~\ref{alg:adc-admm}, it satisfies the following optimality conditions:
	\begin{align*}
		0 &\in \partial f_i(x_i^{k+1}) + \tfrac{1}{2\rho d_i} A_i^T \left( A_i x_i^{k+1} + r_i^{k+1} - t_i^{k+1} \right) \\
		t_i^{k+1} &= \project{\mcf{K}}( A_i x_i^{k+1} + r_i^{k+1} ),
	\end{align*}
	and thus we can write the inclusion above as
	\begin{align*}
		\ifpreprint
			0 &\in \partial f_i(x_i^{k+1}) + \tfrac{1}{2\rho d_i} A_i^T \left( A_i x_i^{k+1} + r_i^{k+1} - \project{\mcf{K}}( A_i x_i^{k+1} + r_i^{k+1} ) \right) \\
		\else
			0 &\in \partial f_i(x_i^{k+1}) + \tfrac{1}{2\rho d_i} A_i^T \big( A_i x_i^{k+1} + r_i^{k+1} \\
			&\phantom{\in \partial f_i(x_i^{k+1}) + \tfrac{1}{2\rho d_i} A_i^T \big(} - \project{\mcf{K}}( A_i x_i^{k+1} + r_i^{k+1} ) \big) \\
		\fi
		&= \partial f_i(x_i^{k+1}) + \tfrac{1}{2\rho d_i} A_i^T \project{\polar{\mcf{K}}}( A_i x_i^{k+1} + r_i^{k+1} ) \\
		&= \partial f_i(x_i^{k+1}) + A_i^T y_i^{k+1},
	\end{align*}
	where the first equality follows from the Moreau decomposition \cite[\Thm~6.30]{Bauschke:2011}, and the second from step~\ref{alg:adc-admm:y} of \Alg~\ref{alg:adc-admm}.
	From the definition of $w^k$, we have
	\begin{align*}
		w^{k+1} = \sum_{i\in\mcf{N}} t_i^{k+1} = \sum_{i\in\mcf{N}} \project{\mcf{K}}( A_i x_i^{k+1} + r_i^{k+1} ) \in \mcf{K},
	\end{align*}
	which means that \eqref{eqn:optimal-cond:w} and \eqref{eqn:optimal-cond:stat} are satisfied in each iteration $k$ by construction.

	Using the Moreau decomposition again, we have
	\ifpreprint
		\begin{align*}
			t_i^{k+1} &= \project{\mcf{K}}( A_i x_i^{k+1} + r_i^{k+1} ) \\
			&= A_i x_i^{k+1} + r_i^{k+1} - \project{\polar{\mcf{K}}}( A_i x_i^{k+1} + r_i^{k+1} ) \\
			&= A_i x_i^{k+1} - b_i + \rho \sum_{j\in\mcf{N}_i} (y_i^k+y_j^k) - p_i^{k+1} - 2 \rho d_i y_i^{k+1},
		\end{align*}
	\else
		\begin{align*}
			t_i^{k+1} &= \project{\mcf{K}}( A_i x_i^{k+1} + r_i^{k+1} ) \\
			&= A_i x_i^{k+1} + r_i^{k+1} - \project{\polar{\mcf{K}}}( A_i x_i^{k+1} + r_i^{k+1} ) \\
			&= A_i x_i^{k+1} - b_i \!+\! \rho \sum_{j\in\mcf{N}_i} (y_i^k \!+\! y_j^k) - p_i^{k+1} \!-\! 2 \rho d_i y_i^{k+1},
		\end{align*}
	\fi
	and therefore
	\ifpreprint
		\begin{align*}
			A_i x_i^{k+1} - b_i - t_i^{k+1} &= p_i^{k+1} + 2 \rho d_i y_i^{k+1} \!-\! \rho \sum_{j\in\mcf{N}_i} (y_i^k+y_j^k) \\
			&= p_i^{k+1} + 2 \rho d_i (y_i^{k+1}-y_i^k) + \rho \sum_{j\in\mcf{N}_i} (y_i^k - y_j^k).
		\end{align*}
	\else
		\begin{align*}
			A_i x_i^{k+1} - b_i - t_i^{k+1} &= p_i^{k+1} + 2 \rho d_i y_i^{k+1} \!-\! \rho \sum_{j\in\mcf{N}_i} (y_i^k+y_j^k) \\
			&= p_i^{k+1} + 2 \rho d_i (y_i^{k+1}-y_i^k) \\
			&\phantom{=} + \rho \sum_{j\in\mcf{N}_i} (y_i^k - y_j^k).
		\end{align*}
	\fi
	Summing the equality above for all $i\in\mcf{N}$ and using \eqref{eqn:app:pi-sum}, we obtain
	\[
		\sum_{i\in\mcf{N}} (A_i x_i^{k+1} - b_i) - w^{k+1} = 2 \rho d_i \sum_{i\in\mcf{N}} (y_i^{k+1} - y_i^k) + \rho \sum_{i\in\mcf{N}} \sum_{j\in\mcf{N}_i} (y_i^k - y_j^k) \to 0.
	\]
	We can characterize $y_i^{k+1}$ as
	\[
		y_i^{k+1} = \tfrac{1}{2\rho d_i} \left( ( A_i x_i^{k+1} + r_i^{k+1} ) - t_i^{k+1} \right) \in \normalCone{\mcf{K}}(t_i^{k+1}),
	\]
	where the inclusion follows from \cite[\Prop~6.47]{Bauschke:2011}, and thus
	\[
		y^\star = \lim_{k\to\infty} y_i^k \in \normalCone{\mcf{K}} \big( \lim_{k\to\infty} t_i^k \big).
	\]
	Due to Lemma~\ref{lem:normal-cone}, we have
	\[
		y^\star \in \normalCone{\mcf{K}} \Big( \lim_{k\to\infty} \sum_{i\in\mcf{N}} t_i^k \Big) = \normalCone{\mcf{K}} \big( \lim_{k\to\infty} w^{k} \big).
	\]
	This concludes the proof.

	\section{Proof of \Prop~\ref{prop:ddc-admm-cvg}}\label{app:prop:ddc-admm-cvg}

	Since \Alg~\ref{alg:ddc-admm} is a direct application of \Alg~\ref{alg:c-admm-sum} to \ref{eqn:dual}, it follows from \cite[\Prop~4]{Mateos:2010} and \cite[\Cor~28.3]{Bauschke:2011} that
	\[
		y_i^k \to y^\star \quad\text{and}\quad z_i^k \to y^\star, \quad \forall i \in \mcf{N},
	\]
	where $y^\star$ is a maximizer of \ref{eqn:dual}.
	We show in the sequel that the iterates  $(x_i^k)_{i\in\mcf{N}}$, $w^k \eqdef \sum_{i\in\mcf{N}} s_i^k$, and $y_i^k$ satisfy optimality conditions~\eqref{eqn:optimal-cond} in the limit.

	Since $x_i^{k+1}$ is a minimizer of the optimization problem in step~\ref{alg:ddc-admm:x} of \Alg~\ref{alg:ddc-admm}, it satisfies the following condition:
	\begin{align*}
		0 &\in \partial f_i(x_i^{k+1}) + \tfrac{1}{\sigma+2\rho d_i} A_i^T \left( A_i x_i^{k+1} + r_i^{k+1} \right) \\
		&= \partial f_i(x_i^{k+1}) +  A_i^T y_i^{k+1},
	\end{align*}
	which means that \eqref{eqn:optimal-cond:stat} is satisfied in each iteration $k$ by construction.
	From step~\ref{alg:ddc-admm:y} of \Alg~\ref{alg:ddc-admm}, we have
	\begin{align*}
		y_i^{k+1} &= \tfrac{1}{\sigma + 2 \rho d_i} ( A_i x_i^{k+1} + r_i^{k+1} ) \\
		&= \tfrac{1}{\sigma + 2 \rho d_i} \big( A_i x_i^{k+1} - b_i - p_i^{k+1} - s_i^{k+1} + \sigma z_i^k + \rho \sum_{j\in\mcf{N}_i} (y_i^k+y_j^k) \big),
	\end{align*}
	and therefore
	\begin{align*}
		A_i x_i^{k+1} - b_i - s_i^{k+1} &= p_i^{k+1} + (\sigma + 2 \rho d_i) y_i^{k+1} - \sigma z_i^k - \rho \sum_{j\in\mcf{N}_i} (y_i^k+y_j^k) \\
		&= p_i^{k+1} + \sigma (y_i^{k+1}-z_i^k) + \rho \sum_{j\in\mcf{N}_i} (2y_i^{k+1} - y_i^k - y_j^k).
	\end{align*}
	Summing the equality above for all $i\in\mcf{N}$ and using \eqref{eqn:app:pi-sum}, we obtain
	\[
		\sum_{i\in\mcf{N}} (A_i x_i^{k+1} - b_i) - w^{k+1} = \sigma \sum_{i\in\mcf{N}} (y_i^{k+1}-z_i^k) + \rho \sum_{i\in\mcf{N}} \sum_{j\in\mcf{N}_i} (2y_i^{k+1} - y_i^k - y_j^k) \to 0.
	\]
	Using the Moreau decomposition in step~\ref{alg:ddc-admm:z} of \Alg~\ref{alg:ddc-admm}, we get
	\begin{equation}\label{eqn:app:prop:ddc-admm-cvg:sigma_zi}
		z_i^{k+1} = \tfrac{1}{\sigma} \project{\polar{\mcf{K}}} \left( s_i^{k+1} + \sigma y_i^{k+1} \right) = \tfrac{1}{\sigma} \left[ \left( s_i^{k+1} + \sigma y_i^{k+1} \right) - \project{\mcf{K}} \left( s_i^{k+1} + \sigma y_i^{k+1} \right) \right],
	\end{equation}
	and thus
	\[
		s_i^{k+1} = \project{\mcf{K}} \left( s_i^{k+1} + \sigma y_i^{k+1} \right) + \sigma (z_i^{k+1}-y_i^{k+1}).
	\]
	From the definition of $w^k$, we obtain
	\[
		\lim_{k\to\infty} w^{k} = \lim_{k\to\infty} \sum_{i\in\mcf{N}} s_i^k = \sum_{i\in\mcf{N}} \project{\mcf{K}} \big( \lim_{k\to\infty} s_i^k + \sigma y^\star \big) \in \mcf{K}.
	\]
	Finally, from \eqref{eqn:app:prop:ddc-admm-cvg:sigma_zi} and \cite[\Prop~6.47]{Bauschke:2011}, we have
	\[
		z_i^{k+1} \in \normalCone{\mcf{K}}\left( \project{\mcf{K}} ( s_i^{k+1} + \sigma y_i^{k+1} ) \right) = \normalCone{\mcf{K}}\left( s_i^{k+1} + \sigma (y_i^{k+1}-z_i^{k+1}) \right).
	\]
	Taking the limit of the inclusion above, we get
	\[
		y^\star \in \normalCone{\mcf{K}}\big( \lim_{k\to\infty} s_i^k \big),
	\]
	and due to Lemma~\ref{lem:normal-cone}, we obtain
	\[
		y^\star \in \normalCone{\mcf{K}} \Big( \lim_{k\to\infty} \sum_{i\in\mcf{N}} s_i^k \Big) = \normalCone{\mcf{K}} \big( \lim_{k\to\infty} w^{k} \big).
	\]
	This concludes the proof.
\fi

\bibliography{refs}

\end{document}